\documentclass{amsart}

\newcommand{\mt}[1]{\mathtt{#1}}

\usepackage{amsmath}
\usepackage{amsthm}
\usepackage{graphicx}
\usepackage{amsfonts}
\usepackage{amssymb}
\usepackage{url}
\usepackage{pstricks-add}

\newcommand{\re}{\mathbb{R}}

\newcommand{\mc}[1]{\mathcal{#1}}

\def\cC{ {\mathcal C} }
\def\cD{ {\mathcal D} }

\def\cH{ {\mathcal H} }

\def\cS{{\mathcal S} }

\def\st{p}

\def\rp{{ \frac{1}{p} }}

\def\ts{{\rp}}

\def\pp{{2 - \rp}}

\def\cS{\mathcal  S}
\def\rp{{\frac{1}{p}}}
\def\demi{{\frac{1}{2} }}

\newcommand{\bdes}{\begin{description}}
\newcommand{\edes}{\end{description}}

\newcommand{\bal}{\begin{align}}
\newcommand{\eal}{\end{align}}

\newcommand{\bnum}{\begin{enumerate}}
\newcommand{\enum}{\end{enumerate}}

\newcommand{\bit}{\begin{itemize}}
\newcommand{\eit}{\end{itemize}}

\newcommand{\bea}{\begin{eqnarray}}
\newcommand{\eea}{\end{eqnarray}}

\newcommand{\be}{\begin{equation}}
\newcommand{\ee}{\end{equation}}

\newcommand{\baray}{\begin{array}}
\newcommand{\earay}{\end{array}}

\newcommand{\bsry}{\begin{subarray}}
\newcommand{\esry}{\end{subarray}}

\newcommand{\bca}{\begin{cases}}
\newcommand{\eca}{\end{cases}}

\newcommand{\bcen}{\begin{center}}
\newcommand{\ecen}{\end{center}}

\newcommand{\bbm}{\begin{bmatrix}}
\newcommand{\ebm}{\end{bmatrix}}

\newcommand{\bmx}{\begin{matrix}}
\newcommand{\emx}{\end{matrix}}

\newcommand{\bpm}{\begin{pmatrix}}
\newcommand{\epm}{\end{pmatrix}}

\newcommand{\btab}{\begin{tabular}}
\newcommand{\etab}{\end{tabular}}

\theoremstyle{plain}
\newtheorem{thm}{Theorem}[section]
\newtheorem{theorem}[thm]{Theorem}

\newtheorem{proposition}[thm]{Proposition}

\newtheorem{lemma}[thm]{Lemma}

\theoremstyle{definition}

\newtheorem{remark}[thm]{Remark}

\makeindex
\usepackage{enumerate}

\def\bs{\bigskip}

\def\ss{\smallskip}

\def\beq{ \begin{equation} }
\def\eeq{ \end{equation} }

\def\bes{\begin{equation*} }
\def\ees{\end{equation*} }

\def\bep{\begin{proof}}
\def\eep{\end{proof}}

\def\ben{\begin{enumerate}}
\def\een{\end{enumerate}}

\def\bet{\begin{theorem}}
\def\eet{\end{theorem}}

\def\bel{\begin{lemma}}
\def\eel{\end{lemma}}
\newcommand{\df}[1]{{\bf{#1}}{\index{#1}}}

\newcommand{\Rg}[1]{{Range({#1}) }}
\makeindex

\begin{document}

\title{Free Semidefinite Representation of Matrix Power Functions}
\author{J. William Helton, Jiawang Nie, Jeremy S. Semko}



\begin{abstract}
Consider the matrix power function $X^p$
defined over the cone of positive definite matrices $\mathcal{S}^{n}_{++}$.
It is known that $X^p$ is convex over $\mathcal{S}^{n}_{++}$ if $p \in [-1,0] \cup [1,2]$
and $X^p$ is concave over $\mathcal{S}^{n}_{++}$ if $p \in [0,1]$.
We show that the hypograph of $X^p$ admits a free semidefinite representation
if $p \in [0,1]$ is rational,
and the epigraph of $X^p$ admits a free semidefinite representation
if $p \in [-1,0] \cup [1,2]$ is rational.
\end{abstract}

\maketitle

\section{Introduction}

Let $\mathcal{S}^n$ be the space of real symmetric $n\times n$ matrices, and
$\mathcal{S}^n_+$ (resp. $\mathcal{S}^n_{++}$) be the cone of
positive semidefinite (resp. definite) matrices in $\mathcal{S}^n$.
For $p \in \re$ , the matrix power function $X^p$ on $\mc{S}^n$
is defined as $X^p= Q^T \Lambda ^p Q$ when this makes sense, with $X=Q^T\Lambda Q$
an orthogonal spectral decomposition.
It is well known (cf.~\cite[pp. 147]{B97}) that
\bit

\item [(1)] \label{it:mpand1p2}
$X^p$ is convex over $\mathcal{S}^n_{++}$
if $p \in [-1, 0]\cup [1, 2]$, and

\item [(2)] \label{it:0p1}
$X^p$ is concave over $\mathcal{S}^n_+$  if $ p \in [0,1]$.

\eit
Here, the concavity and convexity are defined
as usual for functions of matrices.
%for matrix-valued functions.
The goal of this paper is to give a {\it free semidefinite representation}
(i.e., in terms of linear matrix inequalities whose construction
is independent of the matrix dimension $n$)
for the epigraph or hypograph of the matrix power function $X^p$
for a range of rational exponents $p$.

\subsection{Convex and concave matrix-valued functions}\label{sec:concave}

Let $\cD$ be a convex subset of the space of the cartesian product $(\mathcal{S}^n)^g$,
with $g>0$ an integer. A matrix-valued function $f:\cD \to  \mc{S}^n$
is \df{convex} if
\bes
f\big(tX + (1-t)Y\big) \preceq t f(X) +(1-t)f(Y), \quad  \forall \, t \in [0,1]
\ees
for all $X,Y \in \cD$. If $-f$ is convex, we say that $f$ is \df{concave}.
The \df{epigraph} (resp. \df{hypograph}) of $f$ is then defined as
\[
\{(X,Y) \in \cD \times \mc{S}^n:\ f(X) \preceq Y \}  \qquad
(resp. \quad \{(X,Y) \in \cD \times \mc{S}^n :\ f(X)  \succeq Y \} ).
\]
The following is a straightforward but useful fact.
Due to lackness of a suitable reference in case of matrix-valued functions,
we include a short proof here.

\begin{lemma}\label{prop:confuntograph}
Suppose $\cD$ is a convex set.
Then  $f$ is convex over $\cD$ if and only if its epigraph is convex.
Similarly, $f$ is concave over $\cD$ if and only if its hypograph is convex.
\end{lemma}

\begin{proof}
We will prove only the first half of the proposition as the second half clearly follows from the first. \par

($\Rightarrow$) If $(X,W)$ and $(Y,Z)$ are in the epigraph of $f$
and if $t \in [0,1] $, then by the convexity of $f$,
$$ f\big(tX + (1-t)Y \big) \preceq t f(X) +(1-t)f(Y) \preceq t W + (1-t)Z$$
so that $(tX + (1-t)Y, t W + (1-t)Z) $ is in the epigraph of $f$.
%Note that we needed the fact that $\mathcal{G}$ is convex here.
\par

($\Leftarrow$) If $X, Y \in \cD$ and $t \in \mathbb{R}$, then $(X, f(X))$ and $ (Y, f(Y))$ are in the epigraph of $f$.
Since the epigraph is convex, $(tX + (1-t)Y,\  tf(X) + (1-t)f(Y))$ is in the epigraph as well. But this says that $f(tX + (1-t)Y) \preceq  tf(X) + (1-t)f(Y)$.
\end{proof}

In the case that $f(X) \succeq 0$ for all $X \in \cD$,
we are often only interested in the pairs $(X,Y)$ from the hypograph of $f$
with $Y\succeq 0$. Thus, in this case, we slightly abuse terminology and refer to
\[
\{(X,Y) \in \cD \times \mc{S}^n_{+}:\ f(X)  \succeq Y \}
\]
as the \df{hypograph} of $f$. Note that Lemma \ref{prop:confuntograph} remains true with this definition of hypograph.

\subsection{Linear pencils and free semidefinite representation}

Given positive integers $n$ and $g$, let  $(\mc{S}^n)^{g}$
denote the set of $g$-tuples of matrices in $\mc{S}^n$.
Let $\otimes$ denote the standard Kroneker product of two matrices.
%%%%%%%%%%%%%%%%%%%%%
%%%%%%%%%%%%%%%%%%
If $A=(A_0,\dots,A_g )\in (\mc{S}^\ell)^{g+1}$, we define the
 \df{linear pencil} $L_A$,
which acts on $(\mathcal{S}^n)^g$ ($n=1,2,\ldots$) as
\be \label{eq:defLam}
L_A(X): = A_0 \otimes I_n  + \sum_{j=1}^g A_j  \otimes X_j.
\ee
For instance, if
\[
A= \left( \bbm 1 & 2 \\ 2 & 3 \ebm , \bbm 4 & 5 \\ 5 & 6 \ebm,
\bbm 7 & 8 \\ 8 & 9 \ebm \right), \quad \quad X = (X_1, X_2)
\]
with $X_1$ and $X_2$ being $ n \times n$ matrices,
then
\[
L_A(X) = \bbm I_n + 4X_1 + 7X_2 & 2 I_n + 5X_1 + 8X_2 \\
 2I_n + 5X_1 +8X_2 & 3I_n + 6X_1 + 9X_2\ebm
\]
is a $2n \times 2n$ matrix. A \df{monic linear pencil} is a linear pencil with $A_0=I$.

A \df{spectrahedron} in $(\mc{S}^n)^g$ is a set of the form
\be
\cD_{L_A}|_n:= \{  X \in (\cS^n)^{g} :  L_A(X) \succeq 0 \}
\ee
where $L_A$ is a linear pencil.
An inequality of the form $L_A(X) \succeq 0$ is called a
\df{linear matrix inequality (LMI)}.

We now begin the discussion of projected spectrahedra.
For $A \in ({\cS^l})^{g+g'+1}, X \in (\cS^n)^{g}$ and $ W \in (\cS^n)^{g'}$, define
\be \label{eq:defLamtwo}
L_A(X,W): = A_0 \otimes I_n  + \sum_{j=1}^g A_j  \otimes X_j + \sum_{j=g+1}^{g+g'}  A_j  \otimes W_{j-g}.
\ee
We define the projection into the $X$-space as
\be
P_X \cD_{L_A}|_n:= \{X \in (\cS^n)^g: \  \exists W \in (\cS^n)^{g'} \ L_A(X,W) \succeq 0 \}.
\ee

Let $\mc{F}$ be a set in the Cartesian product $\prod_{n=1}^\infty (\mc{S}^n)^g$.
Every element $Z$ of $\mc{F}$ is an $\infty$-tuple in the form
\[
Z = (Z(1), Z(2), \ldots, Z(n), \ldots), \quad
Z(n) \in (\mc{S}^n)^g, \, n = 1, 2, \ldots.
\]
The $n$-th section of $\mc{F}$ is defined as
\be
\mc{F}|_n := \{Z(n) : \  (Z(1), Z(2), \ldots, Z(n), \ldots) \in \mc{F},  \
Z(i) \in (\mc{S}^i)^g, \, i = 1, 2, \ldots  \}.
\ee
A set $\mc{F}$ in $\prod_{n=1}^\infty (\mc{S}^n)^g$ is said to
have a \df{free semidefinite representation} (free SDr) if there exists
a linear pencil $L_A$, in tuples $X$ and $W$,
such that for all $n=1,2,\ldots$
\[
\mc{F}|_n = \{X \in (\mc{S}^n)^g: \,
\exists W \in (\mc{S}^n)^{g^\prime}, \  L_A(X,W) \succeq 0 \}.
\]
In the above, the set $\mc{G} \subseteq
\prod_{n=1}^\infty (\mc{S}^n)^{g} \times \prod_{n=1}^\infty (\mc{S}^n)^{g^\prime}$
defined such that, for all $n=1,2,\ldots,$
\[
\mc{G}|_n = \left\{ (X,W) \in (\mc{S}^n)^{g} \times (\mc{S}^n)^{g^\prime}:\,
L_A(X,W) \succeq 0 \right\}
\]
is called a \df{free LMI lift} of $\mc{F}$.
We emphasize that the key virtue of  free SDr is that
the linear pencil $L_A$ works for all dimensions $n$ of matrix
tuples $X,W$.

\subsection{Contributions}

We consider the matrix power function $f(X):=X^p$.
It is defined over the cone of positive semidefinite matrices for all $p\geq 0$,
and defined over the cone of positive definite matrices for all $p$.
By definition, the epigraph and hypograph of $f$ are naturally sets in $\prod_{n=1}^\infty (\mc{S}^n)^2$.
For convenience, they are respectively denoted as
$\mt{epi}(f)$ and $\mt{hyp}(f)$. Then, for all $n=1,2,\ldots$
\[
\mt{epi}(f)|_n = \{(X,Y) \in (\mc{S}_+^n)^2:  f(X) \preceq Y  \},
\]
\[
\mt{hyp}(f)|_n = \{(X,Y) \in (\mc{S}_+^n)^2:  f(X) \succeq Y\}.
\]
Our main result is the following theorem.

\bet \label{maintheorem}
Let $f(X) = X^p$ be the matrix power function defined over
the cone of positive semidefinite matrices.
If $p \in [0,1]$ is rational, then
the hypograph of $f$ has a free semidefinite representation;
if $p \in [1,2]$ is rational,
then the epigraph of $f$ has a free semidefinite representation. Furthermore, if $p \in [-1,0]$ (restricting the domain to positive definite matrices), then the epigraph of $f$ has a free semidefinite representation.
\eet

As shown in \cite{HM04} and \cite{HKM11},
every polynomial in matrices with convex epigraph
for each dimension  has degree 2 or less.
Also, the sets of symmetric matrices of the form
$$ \cC:= \{ X : f(X) \succeq 0 \}$$
which are convex and bounded all have the form
$ \cC:= \{ X :  L(X) \succeq 0 \}$ for some monic linear pencil $L$.
 As a consequence,
if such a set $\cC$ is semidefinite representable, then it is LMI representable.
These properties also hold when $X$ consists of many matrix variables, for details,
see  \cite{HM12}.
For treatments of rational functions of matrices see \cite{KVV09}.
While we have focused on representing sets with LMI lifts
that is building convex supersets of a given set,
there is no systematic theory of this. There have been clever treatments of special cases
(cf. \cite{OGB02,GO10}).

Most hypographs  and epigraphs in Theorem \ref{maintheorem}
are not spectrahedra.
This can be seen by restricting to   $n=1$ and studying
$\mt{hyp}(f)|_1 $ and $\mt{epi}(f)|_1$. More detail is provided
in \S \ref{sec:norep}.

We should mention the classical SDr literature concerning variables $x$
which are not matrix but scalar variables.
Firstly, there exists a similar result for scalar
power functions $x^p$ by Ben-Tal and Nemirovski \cite{BTN}.
The role of SDr in Optimization appears in Nemirovski \cite{N06}.
For recent advances in SDr, we refer to
Lasserre \cite{Las09a,Las09b,Las10},
Helton and Nie \cite{HN09,HN10}, Netzer \cite{Net10},
Nie \cite{Nie11,Nie12},
Gouveia, Parrilo and Thomas \cite{GPT10}.
For an overview, we refer to the book \cite{BPT}.

\subsection{Ingredients of the proof and guide}
The existence of a free SDr for rational  powers of matrices
is done by a sequence of  constructions which use variables,
 denoted by $W, Z$ and $U$. This takes the remainder
of the paper. We first build a free SDr for
$X^\frac{1}{2}$, %and $X^\frac{1}{3}$,
and then we recursively build constructions for
$X^{1/m}$ for $m \in \mathbb{N}$.
This is done in $\S \ref{sec1/m}$. In $\S$\ref{secallrational}, we build on these in order to construct a free SDr for $X^{s/t}$ for rational $-1 < s/t < 2$. The proof the Theorem \ref{maintheorem} concludes in \S \ref{sec:finalproof}.

\par

\bs

Before continuing, we  collect facts which we will use
throughout the proof :

\begin{lemma} {\rm (L\"{o}wner-Heinz inequality)}
\label{lem:LH}
\cite[pp. 123]{B97}
\\
If $\alpha \in [0,1]$ and $A, B \in \mathcal{S}^n$ such that $A \succeq B \succeq 0$,
then $A^\alpha  \succeq B^\alpha \succeq 0$.
\end{lemma}

Recall the Moore-Penrose pseudoinverse
$C^\dagger$ of a symmetric matrix
$C$ is the symmetric matrx satisfying
$$
C C^\dagger  = C^\dagger  C = P
$$
where
$P$ is the orthogonal  projection  onto  the range space of  C, denoted  $\Rg{C}$.
We refer to \cite{D06}.

\begin{lemma} {\rm (Schur complements)
[Lemma 12.19 in \cite{D06}] }
\label{lem:SC}
\\
If $A, B, C \in \mc{S}^n$, then the block matrix
$ \bbm A  &  B \\ B & C \ebm$ is positive semidefinite  if and only if
$A \succeq BC^{\dagger}B, \ \Rg{B} \subseteq \Rg{C}$ and $C \succeq 0$.

\end{lemma}
Now we list some additional useful facts.
If $A \succeq B$,
 then $M^TAM \succeq M^TBM$ for all matrices $M$.
 If $C \succeq D \succeq 0$
 and $\Rg{C} = \Rg{D}$, then $D^{\dagger} \succeq  C^{\dagger} \succeq 0$.
 Indeed, this is true if $\Rg{C} = \mathbb{R}^n$.
 \footnote{To prove this, we can factorize as
 $C= K^T K$ with $K$ invertible, $D= R^TR$. Then  $K^T K \succeq R^T R$,
 so  $I \succeq K^{-1^T} R^{T} RK^{-1}$ and consequently
 $I \succeq  RK^{-1}  K^{-1^T} R^{T}   $. This implies
 $D^{-1} = R^{-1} R^{-1^T}  \succeq   K^{-1}  K^{-1^T} = C^{-1}  $ .
}
Generally,
we can view $C,D$ as operator mapping into the space $\Rg{C}$.
As a reminder, $X^p$ is only defined for symmetric $X$ such that $X \succeq 0$.
Additionally, all matrices throughout the paper are assumed to be symmetric.

\section{SDr for $X^p$ with $p= 1/m$} \label{sec1/m}
\label{sec:recip}

Throughout this and the next section $p$ will always denote
a rational number. Recall that for each integer $m \geq 0$,
the hypograph of $X^{\frac{1}{m}}$ is defined as
\[
\cH_{1/m} : = \{ (X,Y) :  X^{\frac{1}{m}}
 \succeq Y \succeq 0 \}.
\]
\begin{proposition}
\label{prop:recip}
For all positive integer $m$, the hypograph of $X^{1/m}$
has a free SDr representation.
\end{proposition}
The proof consumes this section and splits in two parts: when $m$ is even and when $m$ is odd.
Each case will use a recursive construction for semidefinite representability.
With $d$ a positive integer, the $p=\frac{1}{2d}$
case relies on the $ p =\frac{1}{d}$ case,
and the $p=\frac{1}{2d+1}$ case relies on the $p =\frac{1}{d+1}$ case.
In other words, if viewed as an algorithm starting with $m$ as the denominator, we move to the case where the denominator is $m/2$ if $m$ is even whereas we move to the case where the denominator is $(m+1)/2$ if $m$ is odd. This will end in the case $m=2$ in finitely many steps.
First we treat the case $m=2$.

\subsection{ $p = 1/2 $ }
\label{sec:1/2}

Consider the hypograph
\[
\cH_{1/2} : = \{ (X,Y) :  X^{1/2} \succeq Y \succeq 0 \}.
\]
Define the free SDr set
\[
\mc{L}_{1/2}: = \{(X,Y) : \exists W \ \bbm X & W \\ W & I \ebm \succeq 0, \   W  \succeq Y \succeq 0 \}.
\]
Clearly, $\mc{L}_{1/2}$ is in the form
\eqref{eq:defLamtwo} of a free SDr set, as we may write
\[
\mc{L}_{1/2}: = \{(X,Y) : \exists W \ L_A(X,Y,W) \succeq 0 \}
\]
where
\[
A= \left( \bbm 0 & 0 & 0 & 0 \\ 0 & 1 & 0 & 0 \\ 0 & 0 & 0 & 0 \\ 0 & 0 & 0 & 0 \ebm ,
               \bbm 1 & 0 & 0 & 0 \\ 0 & 0 & 0 & 0 \\ 0 & 0 & 0 & 0 \\ 0 & 0 & 0 & 0 \ebm,
               \bbm 0 & 0 & 0 & 0 \\ 0 & 0 & 0 & 0 \\ 0 & 0 & -1 & 0 \\ 0 & 0 & 0 & 1 \ebm ,
               \bbm 0 & 1 & 0 & 0 \\ 1 & 0 & 0 & 0 \\ 0 & 0 & 1 & 0 \\ 0 & 0 & 0 & 0 \ebm \right).
\]
For neatness of the paper,
we will not write out  this type of SDr
when it is clear in the context.

\begin{lemma}\label{lem:1/2}
It holds that $\cH_{1/2}= \mc{L}_{1/2}$.
\end{lemma}
\begin{proof} Using Schur complements (Lemma \ref{lem:SC}), we see that
\[
\mc{L}_{1/2}: = \{(X,Y) : \exists W \ X \succeq W^2, \ W  \succeq Y \succeq 0 \}
\]
Clearly, it holds that $\cH_{1/2} \subseteq \mc{L}_{1/2}$ by letting $ W=X^{1/2}$.
Now we prove the reverse containment. By the L\"{o}wner-Heinz inequality,
for all $(X,Y) \in \mc{L}_{1/2}$,
 \[
 X \succeq W^2
  \quad \Rightarrow \quad  X^{1/2} \succeq W
  \quad  \Rightarrow \quad  X^{1/2}  \succeq Y.
\]
\end{proof}

\def\ctH{{\mathcal{\tilde H}}}

\def\xwz{{\bbm X & W \\ W & Z \ebm }}

\def\hG{\hat{G}}

\subsection{ $p=\frac{1}{2d}$ for $d>1$ a positive integer} \label{sec:1/2d}

\begin{lemma}
\label{lem:Ge}
It holds that
\begin{eqnarray*}
\cH_{1/2d} &  = & \{(X,Y) : \exists W \  X \succeq W^2,
\quad  (W,Y) \in\cH_{1/d} \}  \\
& = & \{(X,Y) : \exists W \  \bbm X & W \\
W & I \ebm \succeq 0,  \ (W,Y) \in\cH_{1/d} \}.
\end{eqnarray*}
Consequently, if $\cH_{1/d}$ is a free SDr set, then so is $\cH_{1/2d}$.
\end{lemma}

\begin{proof}
Define
\begin{eqnarray*}
\ctH_{1/2d}&:=& \{(X,Y) : \exists W \  X \succeq W^2,
\quad  (W,Y) \in \cH_{1/d} \} \\
&=& \{(X,Y) : \exists W \  X \succeq W^2,
\quad  W^{1/d} \succeq Y \succeq 0 \}
\end{eqnarray*}
Clearly, it holds that $\cH_{1/2d} \subseteq \ctH_{1/2d}$ by letting $W= X^{1/2}$.
Conversely, if $(X,Y) \in   \ctH_{1/2d}$,
by the L\"{o}wner-Heinz inequality, we can get
\[
 X^{\frac{1}{2d}} \succeq W^{1/d} \succeq Y.
\]
Thus $\cH_{1/2d} = \ctH_{1/2d}$.
\end{proof}

\subsection{ $p=\frac{1}{2d+1}$ for $d$ a positive integer } \label{sec:1/2d+1}

\begin{lemma}
\label{lem:Go}
It holds that
\[
\cH_{\frac{1}{2d+1}} =  \left \{ (X,Y) : \exists (W,Z) \ \  \bbm X & W \\ W & Z \ebm \succeq 0,
\ (W, Z) \in\cH_{\frac{1}{d+1}},  \  Z \succeq  Y \succeq 0 \right \}.
\]
Consequently, if $\cH_{\frac{1}{d+1}}$ is a free SDr set,
then so is $\cH_{\frac{1}{2d+1}}$.
\end{lemma}

\begin{proof}
Let
\[
\ctH_{\frac{1}{2d+1}} : =  \left \{ (X,Y) : \exists (W,Z) \ \
  \bbm X & W \\ W & Z \ebm \succeq 0, \quad
  \ (W, Z) \in\cH_{\frac{1}{d+1}},  \ \  Z \succeq  Y \succeq 0
  %
  % W^{\frac{1}{d+1}}
%\succeq Z \succeq  Y \succeq 0
\right \}.
\]
Note that
\[
\ctH_{\frac{1}{2d+1}} : =
 \{ (X,Y) : \exists (W,Z) \ \  X \succeq W Z^{\dagger} W, \ \Rg{W} \subseteq  \Rg{Z} ,
  \  W^{\frac{1}{d+1}} \succeq Z \succeq  Y \succeq 0 \}
\]
by Lemma \ref{lem:SC}.
The fact that
$\Rg{W} \subseteq \Rg{Z}$ and
  $ W^{\frac{1}{d+1}} \succeq Z \succeq  Y \succeq 0
  $
imply $\Rg{W} = \Rg{Z} $.
Clearly, it holds that $\cH_{\frac{1}{2d+1}} \subseteq \ctH_{\frac{1}{2d+1}}$ by letting $Z=X^{\frac{1}{2d+1}}$ and $W=X^{\frac{d+1}{2d+1}}$.
Now we prove that $\ctH_{\frac{1}{2d+1}} \subseteq\cH_{\frac{1}{2d+1}}$.
Suppose $(X,Y) \in \ctH_{\frac{1}{2d+1}} $.
Note that
\[
W^{\frac{1}{d+1}} \succeq Z \succeq 0
  \quad \Rightarrow \quad  Z^{\dagger} \succeq (W^{\frac{1}{d+1}})^\dagger
  \quad  \Rightarrow \quad
W^{\frac{1}{d+1}} Z^{\dagger} W^{\frac{1}{d+1}} \succeq  W^{\frac{1}{d+1}}.
\]
(The first implication uses the fact  $\Rg{W} = \Rg{Z} $.)
Then it holds that
\[
X \succeq WZ^{\dagger}W =
W^{\frac{d}{d+1}} \big( W^{\frac{1}{d+1}} Z^{\dagger} W^{\frac{1}{d+1}} \big) W^{\frac{d}{d+1}}
\succeq  W^{\frac{2d+1}{d+1}}.
\]
By the L\"{o}wner-Heinz inequality, one gets
\[
X^{\frac{1}{2d+1}} \succeq
W^{\frac{1}{d+1}} \succeq Z \succeq Y.
\]
So, the lemma is true.
\end{proof}

\subsection{Proof of Proposition \ref{prop:recip} }

Given $m$,
the recursions in the lemmas above reduce $\ctH_{1/m}$
having a free SDr representation to $\ctH_{1/\tilde{m}}$ having a free SDR representation for successively smaller $\tilde{m}$.
For example, if $p= 1/14$, then the recursion is $1/14, 1/7, 1/4,1/2$.
This terminates in $m=2$.
Finally, we saw that the hypograph of $X^\demi$
has a free SDr representation, as we have shown earlier.
\qed

\section{SDr for $X^p$ with  $-1<p<2$ rational} \label{secallrational}

The next stage of the proof of Theorem \ref{maintheorem}
 is slightly more involved than the previous $X^{1/m}$ stage.
 Though there are similarities, the recursion steps are not as obvious. For this reason,
we explicitly formulate a recursion defining free SDr sets $\ctH_p$
followed by showing these sets are actually equal to the hypographs
 $$
\cH_{p} : = \{ (X,Y) :  X^p \succeq Y \succeq 0 \}
$$
\index{$\cH_{p}$}for $0 < p < 1$; see \S \ref{sec:Gp}.
After that, it is relatively easy to broaden the range of $p$ to
$-1<p<2$. In particular, we show (in \S \ref{sec:Ep}) that the epigraph
$$\mc{E}_p : = \{ (X,Y) :  X^p \preceq Y, X \succeq 0 \}$$
\index{$\mc{E}_{p}$}is free SDr for $1<p<2$ and free SDr for $-1<p<0$ .

\subsection{$\cH_p$ for $0 < p < 1$ is free SDr}
\label{sec:Gp}

\subsubsection{Preliminaries on  rational  numbers
$0<p<1$}

Define
$$p':= 2- \rp.$$
Clearly, $ 0< p' < 1$ if and only if $1/2<p<1.$
In particular,
$p' = 1/2 $ if and only if $p=2/3$,
$0<p' <1/2 $ if and only if $1/2 <p< 2/3$.\\
\begin{lemma}
\label{lem:pFacts}
 For  $1/2<p<1 $ we have
\ben
\item
 $p' < p$,
% Pf Plot with Mma
\iffalse
 \item Useless is
 $$p'' = 2 - 1/p' = 2 - \frac{p} { 2p -1} =  \frac{2(2p-1) -p}{2p-1}
 = \frac{3p -2}{2p-1}
 =  \frac{3p -2}{2p-1}
 $$
 \fi
 \item
% if $1/2< p <1$, then
the denominator of $p' < $ the denominator of $p$,
%because, denoting $p=s/t$ we have
\item
%if $1/2< p <1$, then
the numerator of $p' < $ the numerator of $p$.
%because, denoting $p=s/t$ we have

\een
\end{lemma}
\begin{proof}
(1): Trivial calculation.\\
(2) and (3):
Denote $p=s/t$ with $t < 2 s < 2t$ and $s,t$ relatively prime.
We have
$$p'= 2 - t/s =  (2 s- t)/s = (s - (t-s))/s$$
with (2) saying $s<t$ and (3) holding because $(t-s)>0 $.
\end{proof}
Suppose $0 <p<1/2$.
There exists an integer $d$ satisfying
$$ 1/2 \leq dp <1;$$
         % since $t + \eps = d (2 s) < 2 t $
let $d(p)$ \index{$d(p)$} denote the smallest such $d$.
%\footnote{Any such $d$ works. We just chose the minimum to make the choice
%unique.}
Clearly, the denominator of $p \geq $ the denominator of $ d(p) p$.

\bs

\subsubsection{Construction of the sequence of rational $p_i$ for $0<p<1$}
\label{sec:buildchain}

We show that for all rational $p\in (0,1)$,
there is a set
$\cS(p) := \{p_0, p_1 , p_2 , \dots, p_m = \demi \}$ \index{$\cS(p)$}
of rational numbers
with each $ p_i \in (0,1)$ such that
\ben
\item[(a)]
$p_0=p$,
%\item
% denominator $p_i \geq$ denominator  $ p_{i+1}$.
\item[(b)]
$\cH_{p_{i-1}}$ is the intersection of a free SDr set and $\cH_{p_i}$.
\een
\ss

First, we show how to construct the set $\cS(p)$.
If $p_i=1/2$, the list terminates.
Otherwise, define $p_{i+1}$ as follows
\ben
\item
\label{it:pless}
 if $0<  p_i < 1/2$, then: $p_{i+1}:= p_i d(p_i)$.

\item
\label{it:pmore}
if $1/2 <p_i<1$, then:  $p_{i+1}:= {2- \frac{1}{p_i}} $.
\een

\bs

\noindent
{\bf Example} \ \
Consider $p_0=7/11$. \\
A. \ Use (2) to get  $p_1  =   2 - 11/7= 3/7$.\\
B. \  Use (1):
we have  $d(p_1)=2$, so  $p_2= 6/7$. \\
C. \  Use (2) to get $p_3= 2- 7/6= 5/6$ and again
to get $p_4= 4/5$ and again to get $p_5= 3/4$
and again $p_6= 2/3$ and again $p_7= 2- 3/2=1/2$. Stop.

\bs

\begin{lemma}
The procedure of constructing $\cS(p)$ as above stops in a finite number of steps.
\end{lemma}

 \begin{proof}
 Everytime \eqref{it:pmore} is invoked
 the denominator strictly decreases.
 Also  \eqref{it:pless} never increases the denominator.
 This is shown by Lemma \ref{lem:pFacts}.
 Immediately after \eqref{it:pless} is applied, \eqref{it:pmore} is
 always applied. Hence
 the denominators decrease
 until one obtains $p_m$ whose denominator
 is 2. Since $0< p_m <1$, we get $p_m=1/2$ and the recursion stops.
 \end{proof}

\begin{remark}
The number of steps $k$ that the above procedure requires
is at most two times the denominator of $p$.
\end{remark}

\subsubsection{The recursion on  $\cH_p$ for $0<p<1$}
We now show that if $p_{i-1}, p_i$ are on the list $\cS(p_0)$, then $\cH_{p_{i-1}}$ is free SDr provided $\cH_{p_i}$ is.
This fact follows from the lemmas below.

\begin{lemma}
\label{lem:pGmore}
Suppose $ 1/2< p<1$.
Then $\cH_p =\ctH_p$
where  $\ctH_p$  is defined to be
$$
 \ctH_p L=  \left\{ (X,Y) : \  \exists (W,Z) \in \cH_{\pp}, \ \
 \  \xwz \succeq 0,
 \    \   \  W\succeq Y \succeq 0 \right\}.
$$
\end{lemma}
\begin{proof}
The set $\ctH_p$ can be equivalently written as
{\footnotesize
$$
\{ (X,Y) : \  \exists (W,Z) \ \
W^{\pp} \succeq  Z \succeq 0, \
\Rg{W} = \Rg{Z}, \
X \succeq W Z^{\dagger} W,
 \  \   W \succeq Y \succeq 0 \}.
 $$
}
It holds that
$\cH_{\st}
 \subseteq \ctH_p$
 by letting $W=X^{\st}$ and $Z=X^{2p-1}$.
Now we prove that $\tilde{\mc{H}}_p \subseteq\cH_{\st}$ .
%Here
%$ \mc{\tG }_{\st}^\circ :=   \{ (X,Y) \in  \mc{\tG }_{\st}: \ Y \succ 0 \}$.
Start with  $(X,Y) \in \ctH_p$, then there are $W,Z$
with $\Rg{Z} = \Rg{W}$, satisfying
\[
 W^{\pp } \succeq Z \succeq 0
  \quad \Rightarrow \quad  Z^{\dagger} \succeq (W^{\pp})^\dagger \quad  \Rightarrow \quad  W^{\pp } \;  Z^{\dagger} \; W^{\pp }
  \succeq W^{\pp}.
\]
Thus
\[
X \succeq WZ^{\dagger}W =
 W^{\ts-1}  \big( W^{\pp} \;  Z^{\dagger \; } W^{\pp } \big)
 W^{\ts-1}
\succeq W^{ \ts}.
\]
By the L\"{o}wner-Heinz inequality, one gets
\[
X^{\st}\succeq W \succeq Y \succeq 0.
\]
Hence, $\tilde{\cH}_p = \cH_{\st}$.
\end{proof}

\begin{lemma}
\label{lem:pGless}
  Suppose $ 0< p<1/2$. Let
$$
\ctH_{\st} : = \{ (X,Y) : \exists W \ \ \ X^{1/d(p)} \succeq  W,  \ \
W^{d(p) \st} \succeq  Y \succeq 0 \} \\
$$
$$
    \quad  \quad \quad    = \{ (X,Y) : \exists W  \ \
          (X,W) \in \cH_{1/d(p)} , \ \
                         (W, Y) \in \cH_{d(p) \st}  \}.
%W^{d \st} \succeq  Y \succeq 0 \}
$$
Then
$\cH_{\st} = \ctH_{\st}$.
\end{lemma}

\begin{proof}
Observe that $\cH_{\st} \subseteq {\ctH}_{\st}$,
by letting $W=X^{1/d(p)}$.
Now we prove the reverse containment. From
the L\"{o}wner-Heinz inequality, for all $(X,Y) \in {\ctH}_{\st}$,
\[
X^{\st} =  \Big( X^{1/{d(p)}} \Big)^{d(p)\st} \succeq
\Big( W \Big)^{d(p)\st} \succeq Y,
\]
because $d(p) p < 1$.
\end{proof}

\subsubsection{$\cH_p$ is free SDr for $0<p<1$}
\label{sec:pfmain}
Consider the list $\cS(p)$ of rational numbers constructed in \S \ref{sec:buildchain}.
Proposition \ref{prop:recip} and Lemmas \ref{lem:pGmore} and \ref{lem:pGless}
tell us that
$\cH_{p_{i-1}}$ is free SDr if $\cH_{p_{i}}$ is. By \S \ref{sec:1/2},
$\cH_{1/2}$ is free SDr and thus  $\cH_{\st_j} $ is free SDr
for all $0 \leq j \leq m$. In particular
$\cH_{p_0}$ is free SDr where $p_0 = p$.
This completes the proof that $\cH_p$ is a free SDr set for all $0<p<1$.

\subsection{Broadening the range of $p$ to  $-1<  p < 2$}
\label{sec:Ep}

\subsubsection{$1<  p < 2$}\label{sec12}
Consider the epigraph
\[
\mc{E}_{p} : = \{ (X,Y) :  X^p \preceq Y, X \succeq 0 \}.
\]
Define the free SDr set
\[
\tilde{\mc{E}}_{p} : = \left \{ (X,Y) : \exists Z \  \bbm Y & X \\ X & Z \ebm \succeq 0, \ (X, Z) \in \cH_{2-p} , \  X \succeq 0 \right \}
\]
By \S \ref{sec:pfmain},
$\cH_{2-p}$ is free SDr (since $0<2-p<1$) .

\begin{lemma}
It holds that $\mc{E}_{p}=\mc{\tilde{E}}_{p}$ for $1 < p < 2$.
\end{lemma}
\begin{proof}
First note that
\[
\mc{\tilde{E}}_{p}  = \left \{ (X,Y) : \exists Z \  \bbm Y & X \\ X & Z \ebm \succeq 0,  \ X^{2-p} \succeq Z \succeq 0, \   X \succeq 0 \right \}
\]
\[ \quad \quad \quad \quad \quad \quad \ \ \
 = \{ (X,Y) : \exists Z \ \ Y \succeq X Z^{\dagger} X, \ \Rg{Z} = \Rg{X},
   \  X^{2-p} \succeq Z \succeq 0, \  X \succeq 0 \}
\]
by Lemma \ref{lem:SC}.
Clearly, it holds that $ \mc{E}_{p} \subseteq \mc{\tilde{E}}_{p}$ by letting $Z=X^{2-p}$ .
Now we prove that $\mc{\tilde{E}}_{p} \subseteq \mc{E}_{p} $. From
the L\"{o}wner-Heinz inequality, for all $(X,Y) \in \mc{\tilde{E}}_{p}$,
\[
 X^{2-p} \succeq Z \succeq 0
  \quad \Rightarrow \quad  Z^{\dagger} \succeq (X^{2-p})^\dagger
  \quad  \Rightarrow \quad  X Z^{\dagger} X \succeq X^p.
\]
Thus, $Y \succeq X^p$
\end{proof}

\subsubsection{$-1<  p < 0$}\label{sec-10}
Consider the epigraph
$$
\mc{E}_{p} : = \{ (X,Y) :  Y  \succeq  X^p  \succ 0 \} .
$$
Define the free SDr set
\[
\mc{\tilde{E}}_{p} : = \{ (X,Y) : \exists Z \   (X,Z) \in \cH_{-p} ,
\ \bbm Z & I \\ I & Y \ebm \succeq 0,\  X \succ 0 \}.
\]
\begin{lemma}
It holds that $\mc{{E}}_{p}=\mc{\tilde{E}}_{p}$ for $-1 < p < 0$.
%
%$\mc{E}_{p} = \{ (X,Y) \in \mc{\tilde{E}}_{p} : \  X, Y \succ 0 \}$
\end{lemma}
\begin{proof}
Note that in this case
%
%\[
%\mc{E}_{p}  = \{ (X,Y) : \ X^p \preceq Y, X \succeq 0 \}  =  \{  X^{-p} \succeq Y^{-1}, X \succeq 0, Y \succeq 0 \}
%\]
%because the function $X^p$ for $-1<p<0$ is only defined for $X \succ 0$. Thus, in the left-hand representation of $\mc{E}_{p}$, $X^p \succ 0$ and therefore $Y \succ 0$. Likewise, in the right-hand representation, $Y^{-1}$ exists so $Y \succ 0$  and so $X^{-p} \succ 0 \Rightarrow X \succ 0$. In summary
\[
\mc{E}_{p}  = \{ (X,Y) : \ X^p \preceq Y, X \succ 0 \}  =  \{  X^{-p} \succeq Y^{-1}, X \succ 0, Y \succ 0 \}.
\]
Now by \S \ref{sec:pfmain}
we have that $\cH_{-p}$ is free SDr ($0<-p<1$) and that
\[
\mc{\tilde{E}}_{p} : = \{ (X,Y) : \  \exists Z \ \   X^{-p} \succeq Z, \  \bbm Z & I \\ I & Y \ebm \succeq 0,\  X \succeq 0 \}
\]
Clearly, $\mc{E}_{p} =\{ (X,Y) \in \mc{\tilde{E}}_{p} : \ X, Y \succ 0 \}$ (Letting $Z=Y^{-1}$ on one hand and using Schur complements on the other).
 \end{proof}

\subsection{Proof of Theorem \ref{maintheorem}}\label{sec:finalproof}
 Now we put the results together for rational numbers $p$ in $-1<p<2$. From \S\ref{sec:pfmain}, we have that the hypograph of $X^p$ ($0<p<1$) is free SDr with the domain $\mathcal{S}_{+}^n$. From \S\ref{sec12}, the epigraph of $X^p$ ($1<p<2$) is free SDr again with the domain $\mathcal{S}_{+}^n$. Shrinking the domain to $\mathcal{S}_{++}^n$, \S\ref{sec-10} shows the epigraph of $X^p$ ($-1<p<0$) is free SDr.
%Interpreting $X^p$ as a function on the smaller space $\mathcal{S}_{++}^n$  in all of the above cases, we have that for $-1 <p <0$ and $1<p<2$, $X^p$ has a free SDr epigraph and for $0<p<1$,
%$X^p$ has a free SDr hypograph.
This proves Theorem \ref{maintheorem}.

\section{Matrix concavity in several variables}

One would attempt to generalize the abo ve results to
the case for symmetric multivariate matrix functions. A natural case to consider is the root function
\[
q(X) = \Big( X_g^{p_g/2}\cdots X_1^{p_1/2}   X_0^{p_0} X_1^{p_1/2}\cdots X_g^{p_g/2}  \Big)^{1/k}
\]
with $k\geq p_0 + p_1 + \cdots + p_g$ and $p_j \in \mathbb{Q}$ (i.e. we are taking a root of a simple symmetric multivariable polynomial) where $q$ is defined on $g$-tuples of positive semidefinite symmetric matrices (i.e. for $X=(X_1, \dots, X_g)  \in (\mathcal{S}^n_+)^g$).
Unlike the univariate case,  even the simplest  function
 of this kind is  not concave.
For instance, the set
\[
\{(X_0, X_1) : (X_1X_0X_1)^{1/3} \succeq I  \}
\]
which is the same as the set
 \[ \{ (X_0, X_1) : X_1X_0X_1 \succeq I \}
\]
is not convex. If it were, then fixing $$X_0=\begin{pmatrix} 4 & 0 \\ 0 & 1 \end{pmatrix}^{-1}, \quad  A =\begin{pmatrix} 4 & 0 \\ 0 & 1 \end{pmatrix} $$ and letting $X_1=X$
would imply that the set
\[
Q= \left\{ X : X^2 \succeq A, X \succeq 0 \right\}
\]
is convex. However, letting
$$X_1 = \begin{pmatrix} 2 & 0 \\ 0 & 1 \end{pmatrix},  \quad X_2 = \begin{pmatrix} 3 & 1 \\ 1 & 133/64 \end{pmatrix},$$
we have that $X_1, X_2 \in Q$ but that $Z=(X_1+X_2)/2 \not\in Q$.
This is because the matrix
%
% $$E:= Z^2 - A = \begin{pmatrix} 2.5 & 2.01953 \\ 2.01953 & 1.61871 \end{pmatrix}$$
%
\[
E:= Z^2 - A =
\bpm    5/2 & 517/256 \\ 517/256 & 26521/16384 \epm
\]
is not positive semidefinite (its determinant is $-2079/65536<0$).
Thus, our natural generalization of the single variable root function does not
preserve concavity when more variables are added.

\iffalse

n = 2;
A = randn(n); A=A'*A; Ar =sqrtm(A);
B = randn(n); B=A+B'*B; Br=sqrtm(B);
C = (Ar+Br)^2/4-A;
eig(C),

A = [ 4.1830   -0.0049 ;   -0.0049    0.0057],

Ar =[ 2.045237067266954,  -0.002310557783379 ;  -0.002310557783379,   0.075462979816130],

Br = [   3.0063   -0.3325 ;  -0.3325    0.4563],

syms t
X_1 = t*[ 2 0; 0 1]; X_2 = t*[3 1; 1  133/64];
A=t^2*[4 0; 0 1]; Z = (X_1+X_2)/2;
E=Z^2-A;
pretty(E/t^2),
latex(E/t^2),

\fi

\section{An SDr is required}
\label{sec:norep}

\def\tp{{\tilde p}}
\def\tq{{\tilde q}}

Most hypographs  and epigraphs in Theorem \ref{maintheorem}
are not spectrahedra.
This can be seen by studying
$\mt{hyp}(f)|_1 $ and $\mt{epi}(f)|_1$ and applying  \cite{HV07}
which characterizes exactly which sets in $\mathbb{R}^2$
are the solution set to some LMI.

Set $n=1$, $p=\frac{s}{t} $ for coprime  integers $s,t$, and $f(X)=X^{s/t}$. In the remainder of this section, we will use lowercase $x$ and $y$ to reinforce that we are working in commuting variables.
Note that we can write $\mt{hyp}(f)|_1$ as
\begin{eqnarray*}
\mt{hyp}(f)|_1
 &= &\{ (x,y) : \  f(x) \geq  y, \    x  \geq 0, \ y\geq  0 \} \\
     &= &\{ (x,y) : \  q(x,y)  \geq  0, \    x  \geq 0, \ y\geq  0 \} \\
     &=& closure \ of \ component \ of \ (1,1/2) \ of \
     \{ (x,y) : \  q(x,y)  > 0   \}
\end{eqnarray*}
where $q$ is the polynomial
$q(x,y) = (x^s-y^t) y  $.
Such a $q$ is a minimum degree defining polynomial of
$\mt{hyp}(f)|_1 $ and has degree $1+s \vee t$.

The set  $\mt{hyp}(f)|_1$ passes the \cite{HV07} line test  if almost any line $\ell$ through
the point $(1, \frac1 2)$ intersects
 the Zariski closure of boundary of
$\mt{hyp}(f)|_1$ (denoted $Z_f$)  in $1 + s \vee t$ points.
Passing the line test is equivalent to $\mt{hyp}(f)|_1$
being the solution set to some LMI by Theorem 2.2 of \cite{HV07}.

%$q(X,Y) = (Y^t - X^s) Y  $.

Suppose $s\geq 0$ and $t > 0$. Then it is easy to see
$Z_f$ equals $\{y=0\}$ union
\ben
\item
 the set
$q_{t \ odd}:= \{(x,(x^{s})^{1/t}) : x \in \mathbb{R} \}$
if $t$ is  odd,
\item
 the set
$q_{t \ even}:= \{((y^{t})^{1/s},y) : y \in \mathbb{R} \}$
if $t$ is even (so $s$ is odd)
\een

Assume $p= \frac{s}{t} \in [0,1]$. We will analyize three possible cases which depend on the parity of $s$ and $t$. The general shapes of $q_{t \ odd}$ and $q_{t \ even}$ are shown in the figure where the blue dot represents the point  $(1, \frac1 2)$.

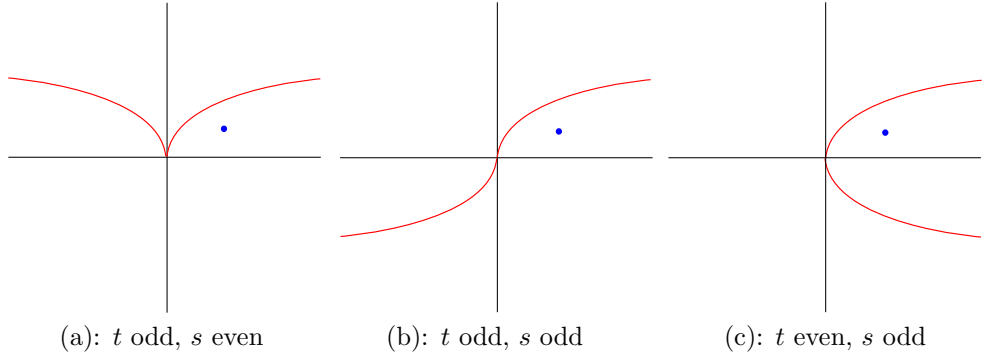
\begin{figure}[h]
\psset{xunit=.5pt,yunit=.5pt,runit=.5pt}
\begin{pspicture}(736.38153076,234.99250793)
{
\newrgbcolor{curcolor}{0 0 0}
\pscustom[linewidth=0.78496858,linecolor=curcolor]
{
\newpath
\moveto(120.13591562,234.61919744)
\lineto(120.13591562,0.00333722)
}
}
{
\newrgbcolor{curcolor}{0 0 0}
\pscustom[linewidth=0.78496855,linecolor=curcolor]
{
\newpath
\moveto(370.01123533,234.99160273)
\lineto(370.01123533,0.3757565)
}
}
{
\newrgbcolor{curcolor}{0 0 0}
\pscustom[linewidth=0.78496855,linecolor=curcolor]
{
\newpath
\moveto(487.44513369,117.3112733)
\lineto(251.07657595,117.3112733)
}
}
{
\newrgbcolor{curcolor}{0 0 0}
\pscustom[linewidth=0.78496855,linecolor=curcolor]
{
\newpath
\moveto(236.4430595,117.6803302)
\lineto(0.0745068,117.6803302)
}
}
{
\newrgbcolor{curcolor}{1 0 0}
\pscustom[linewidth=0.8690325,linecolor=curcolor]
{
\newpath
\moveto(120.06594942,118.18588406)
\curveto(125.23425203,168.67193103)(236.08910386,177.05750101)(236.08938153,177.05750566)
}
}
{
\newrgbcolor{curcolor}{1 0 0}
\pscustom[linewidth=0.8837778,linecolor=curcolor]
{
\newpath
\moveto(119.13839742,118.41417271)
\curveto(113.83187249,169.26786831)(0.01211941,177.71450056)(0.0118244,177.71450676)
}
}
{
\newrgbcolor{curcolor}{1 0 0}
\pscustom[linewidth=0.8690325,linecolor=curcolor]
{
\newpath
\moveto(369.96256449,117.52191906)
\curveto(375.13086016,168.00796604)(485.98572016,176.39353602)(485.98599782,176.39354066)
}
}
{
\newrgbcolor{curcolor}{1 0 0}
\pscustom[linewidth=0.87782165,linecolor=curcolor]
{
\newpath
\moveto(369.53398314,116.69666287)
\curveto(364.27135844,66.10765718)(251.39305439,57.70498809)(251.39277152,57.70498809)
}
}
{
\newrgbcolor{curcolor}{1 0 0}
\pscustom[linewidth=0.87634846,linecolor=curcolor]
{
\newpath
\moveto(618.37556251,117.08273057)
\curveto(623.63185996,167.5628501)(736.37425844,175.94743475)(736.37455345,175.9474394)
}
}
{
\newrgbcolor{curcolor}{1 0 0}
\pscustom[linewidth=0.88054276,linecolor=curcolor]
{
\newpath
\moveto(617.88194729,116.69537544)
\curveto(623.13805385,65.72898871)(735.87620062,57.26362135)(735.87649563,57.26362135)
}
}
{
\newrgbcolor{curcolor}{0 0 0}
\pscustom[linewidth=0.78496858,linecolor=curcolor]
{
\newpath
\moveto(618.38579562,234.99160344)
\lineto(618.38579562,0.37574322)
}
}
{
\newrgbcolor{curcolor}{0 0 0}
\pscustom[linewidth=0.78496858,linecolor=curcolor]
{
\newpath
\moveto(735.81968956,117.31127057)
\lineto(499.45114156,117.31127057)
}
}
{
\newrgbcolor{curcolor}{0 0 1}
\pscustom[linestyle=none,fillstyle=solid,fillcolor=curcolor]
{
\newpath
\moveto(164.6596656,139.21145785)
\curveto(164.6596656,138.37095973)(164.00743761,137.68960139)(163.2028736,137.68960139)
\curveto(162.39830959,137.68960139)(161.7460816,138.37095973)(161.7460816,139.21145785)
\curveto(161.7460816,140.05195596)(162.39830959,140.7333143)(163.2028736,140.7333143)
\curveto(164.00743761,140.7333143)(164.6596656,140.05195596)(164.6596656,139.21145785)
\closepath
}
}
{
\newrgbcolor{curcolor}{0 0 1}
\pscustom[linewidth=1.82182495,linecolor=curcolor]
{
\newpath
\moveto(164.6596656,139.21145785)
\curveto(164.6596656,138.37095973)(164.00743761,137.68960139)(163.2028736,137.68960139)
\curveto(162.39830959,137.68960139)(161.7460816,138.37095973)(161.7460816,139.21145785)
\curveto(161.7460816,140.05195596)(162.39830959,140.7333143)(163.2028736,140.7333143)
\curveto(164.00743761,140.7333143)(164.6596656,140.05195596)(164.6596656,139.21145785)
\closepath
}
}
{
\newrgbcolor{curcolor}{0 0 1}
\pscustom[linestyle=none,fillstyle=solid,fillcolor=curcolor]
{
\newpath
\moveto(418.0270556,137.21955785)
\curveto(418.0270556,136.37905973)(417.37482761,135.69770139)(416.5702636,135.69770139)
\curveto(415.76569959,135.69770139)(415.1134716,136.37905973)(415.1134716,137.21955785)
\curveto(415.1134716,138.06005596)(415.76569959,138.7414143)(416.5702636,138.7414143)
\curveto(417.37482761,138.7414143)(418.0270556,138.06005596)(418.0270556,137.21955785)
\closepath
}
}
{
\newrgbcolor{curcolor}{0 0 1}
\pscustom[linewidth=1.82182495,linecolor=curcolor]
{
\newpath
\moveto(418.0270556,137.21955785)
\curveto(418.0270556,136.37905973)(417.37482761,135.69770139)(416.5702636,135.69770139)
\curveto(415.76569959,135.69770139)(415.1134716,136.37905973)(415.1134716,137.21955785)
\curveto(415.1134716,138.06005596)(415.76569959,138.7414143)(416.5702636,138.7414143)
\curveto(417.37482761,138.7414143)(418.0270556,138.06005596)(418.0270556,137.21955785)
\closepath
}
}
{
\newrgbcolor{curcolor}{0 0 1}
\pscustom[linestyle=none,fillstyle=solid,fillcolor=curcolor]
{
\newpath
\moveto(664.9486856,136.33427785)
\curveto(664.9486856,135.49377973)(664.29645761,134.81242139)(663.4918936,134.81242139)
\curveto(662.68732959,134.81242139)(662.0351016,135.49377973)(662.0351016,136.33427785)
\curveto(662.0351016,137.17477596)(662.68732959,137.8561343)(663.4918936,137.8561343)
\curveto(664.29645761,137.8561343)(664.9486856,137.17477596)(664.9486856,136.33427785)
\closepath
}
}
{
\newrgbcolor{curcolor}{0 0 1}
\pscustom[linewidth=1.82182495,linecolor=curcolor]
{
\newpath
\moveto(664.9486856,136.33427785)
\curveto(664.9486856,135.49377973)(664.29645761,134.81242139)(663.4918936,134.81242139)
\curveto(662.68732959,134.81242139)(662.0351016,135.49377973)(662.0351016,136.33427785)
\curveto(662.0351016,137.17477596)(662.68732959,137.8561343)(663.4918936,137.8561343)
\curveto(664.29645761,137.8561343)(664.9486856,137.17477596)(664.9486856,136.33427785)
\closepath
}
}
\end{pspicture}

\ \ (a): $t$ odd, $s$ even  \qquad \qquad  \  (b): $t$ odd, $s$ odd \qquad \qquad  \ \ \ (c): $t$ even, $s$ odd
\end{figure}

Case (a): $t$ is odd and $s$ is even. The function $(x^{s})^{1/t}$ is monotone in $x$ for $x > 0$ and is symmetric across the $y-$axis. Any line $\ell$ intersecting
$(1, \frac 1 2)$ with slope less than $-1/2$
intersects the graph of  $(x^{s})^{1/t}$ exactly once because $s/t \in [0,1]$. In other words, $\ell$ intersects $q_{t \ odd}$ once.
Additionally, it intersects $\{y=0\}$ exactly once. In all, $\ell$ intersects $Z_f$  twice, and thus $\mt{hyp}(f)|_1$ is a spectrahedron
only if  $2 = 1+s \vee t $. This means that $s=0$, i.e. $p=0$.

 Case (b): $t$ and $s$ are odd. We have $(x^{s})^{1/t}$ is monotone increasing as a function of $x$ and intersects the origin. Any negatively sloped  line  $\ell$
through
$(1, \frac 1 2)$  intersects the graph of  $(x^{s})^{1/t}$ exactly once, so it intersects the set $q_{t \ odd}$ once.
Additionally, it intersects $\{y=0\}$ exactly once.
Thus it intersects $Z_f$  twice and so if $\mt{hyp}(f)|_1$ is a spectrahedron
then  $2 = 1+s \vee t $. In this case, $s=1$ and $t=1$ so that $p=1$.

Case (c):  $t$ is even and $s$ is odd. For $y >0 $ the function $(y^{t})^{1/s}$ is monotone and because $s$ is odd, it is symmetric across the $x-$axis.
Any  line  $\ell$ through $(1, \frac 1 2)$  with negative slope will intersect set $q_{t \ even}$ twice: once clearly when $y$ is positive and once when $y$ is negative, since $t/s > 1$ implies that the slope of the graph in the $4^{th}$ quadrant is increasing.
It also intersects $\{y=0\}$ exactly once.
Thus it intersects $Z_f$ three times.
We conclude that  $\mt{hyp}(f)|_1$ is a spectrahedron
only if $3 = 1+s \vee t$; that is, $p=1/2$.

We conclude that when $p=s/t \in [0,1]$, the only nontrivial case in which the hypograph of $f(X)=X^{s/t}$ can be a spectrahedron is when p=1/2. While we have analyzed just $p \in [0,1]$,
the other cases behave similarly.

\section{Acknowledgments}
Helton was partially funded by NSF grants
DMS-0700758, DMS-0757212, DMS-1160802, DMS-1201498
and the Ford Motor Co.,
Semko by   DMS-0700758,
%frg DMS-0757212,
DMS-1160802.
Nie was partially supported by the NSF grant DMS-0844775.

%\newpage
%
%\centerline{NOT FOR PUBLICATION}
%
%\tableofcontents
%
%
%
%\printindex

%\bibliographystyle{plain}
%\bibliography{myrefs}

\begin{thebibliography}{OHMP69}


\bibitem[BTN]{BTN}
Ben-Tal, A.; Nemirovski, A.:
{\it Lectures on Modern Convex Optimization: Analysis, Algorithms, and Engineering
Applications}. MPS-SIAM Series on Optimization, SIAM, Philadelphia, 2001.

 \bibitem[BPT]{BPT} Blekherman, G; Parrilo, P.; Thomas R.
  Semidefinite Optimization and Convex Algebraic Geometry,
  vol 13
  SIAM optimization series pp362
 2012


\bibitem[B97]{B97}
Bhatia,
  Rajendra:
   Matrix Analysis,
    1997,
   Springer-Verlag,
    New York,
Graduate Studies in Mathematics,
169




\bibitem[D06]{D06}
Dym,
   Harry:
    Linear Algebra in Action,
  2006,
 AMS,
    New York,
Graduate Studies in Mathematics,78
 p1 - 539





\bibitem[GO10]{GO10}
{Graham, Matthew R.; de Oliveira, Maur\'{\i}cio C.}:
 {Linear matrix inequality tests for frequency domain
  inequalities with affine multipliers},
{Automatica},
Num {5}, {897-901},  {46},
{2010},



\bibitem[GPT10]{GPT10}
Gouveia, J.; Parrilo, P.A.; Thomas, R.:
Theta Bodies for Polynomial Ideals.
{\it SIAM J. Optim.}, Vol. 20, No. 4, pp. 2097-2118, 2010.




\bibitem[HKM11]{HKM11}
Helton, J.W.; Klep, I.; McCullough, S.:
Convexity and Semidefinite Programming in dimension-free matrix unknowns,
In: {\em Handbook of Semidefinite, Cone and Polynomial Optimization}
edited by M. Anjos and J. B. Lasserre, Springer-Verlag, 2011

 \bibitem[HM04]{HM04} Helton, J.W.; McCullough, S.:
Convex
 noncommutative polynomials
have degree two or less, {\em SIAM J. Matrix Anal. Appl.} {\bf 25} (2004)
1124--1139


\bibitem[HM12]{HM12}
Helton, J.W.; McCullough, S.:
Every free basic convex semi-algebraic set has an LMI representation,
Annals of Math 2012




\bibitem[HN09]{HN09}
Helton, J.W.; Nie, J.:
Sufficient and Necessary
Conditions for Semidefinite Representability of Convex Hulls and Sets.
{\it SIAM Journal on Optimization} 20(2009), no.2, pp. 759-791.


\bibitem[HN10]{HN10}
Helton, J.W.; Nie, J.:
Semidefinite representation of convex sets.
{\it Mathematical Programming},
Vol. 122, No.1, pp.21-64, 2010.



\bibitem[HV07]{HV07}
Helton, J.W.; Vinnikov V.
Linear matrix inequality representation of sets.
{\it Comm. Pure Appl. Math.} 60 (2007), No. 5, pp. 654-674.


\bibitem[KVV09]{KVV09}
Kalyuzhnyi-Verbovetski\u\i{}, D.; Vinnikov, V.:
Singularities of rational functions and minimal factorizations: the noncommutative and the commutative setting, {\em Linear Algebra Appl.} {\bf 430}
(2009) 869--889


\bibitem[Las09a]{Las09a}
Lasserre, Jean B.:
Convex sets with semidefinite representation.
{\it Mathematical Programming}, Vol.~120, pp. 457--477, 2009.


\bibitem[Las09b]{Las09b}
Lasserre, Jean B.:
Convexity in SemiAlgebraic Geometry and Polynomial Optimization.
{\it SIAM J. Optim.} \, Vol.~19,   pp. 1995--2014, 2009.





\bibitem[Las10]{Las10}
Lasserre, Jean B.:
 Moments Positive Polynomials and Their
Applications \
Imperial College Press 2010 pp361


\bibitem[N06]{N06}
Nemirovski,  A.:  Advances in convex optimization: conic
programming. {\it Plenary Lecture,  International Congress of
Mathematicians (ICM)}, Madrid, Spain, 2006.


\bibitem[Net10]{Net10}
Netzer, Tim: On Semidefinite Representations of Non-closed Sets.
{\it Linear Algebra and its Applications} 432, 3072-3078 (2010).


\bibitem[Nie11]{Nie11}
Nie, J.:
Polynomial matrix inequality and semidefinite representation.
{\it Mathematics of Operations Research}, Vol. 36, No. 3, pp. 398-415, 2011.



\bibitem[Nie12]{Nie12}
Nie, J.: First Order Conditions for Semidefinite Representations of Convex Sets Defined by
Rational or Singular Polynomials.
{\it Mathematical Programming}, Ser. A, Vol. 131, No. 1, pp. 1-36, 2012.



\bibitem[OGB02]{OGB02}
     de Oliveira, Maur\'{\i}cio C.;Geromel, Jos\'e C.;
Bernussou, Jacques:
              {Extended {$H_2$} and {$H_\infty$} norm characterizations and
controller parametrizations for discrete-time systems},
       {International Journal of Control},
        {9},
        pp {666--679},
        Vol {75},
         {2002}


\end{thebibliography}

\end{document}